\documentclass[reqno,12pt]{amsart}

\usepackage{amsmath, epsfig, cite}
\usepackage{amssymb}
\usepackage{amsfonts}
\usepackage{latexsym}
\usepackage{amsthm}

\newtheorem{theorem}{Theorem}[section]

\theoremstyle{remark}

\numberwithin{equation}{section}

\allowdisplaybreaks

\author{Qinghu Hou}

\author{Haihong He}
\address{DEPARTMENT OF MATHEMATICS, SHANGHAI UNIVERSITY, SHANGHAI 200444, P. R. CHINA}
\email{$^*$ Corresponding author. qh\_hou@tju.edu.cn (Q. Hou), hehaihong5@163.com (H. He), xiaoxiawang@shu.edu.cn (X. Wang).}
\author{Xiaoxia Wang*}
\thanks{This work is supported by Natural Science Foundation of Shanghai (22ZR1424100).}
\title[Ramanujan-inspired series for $1/\pi$  involving harmonic numbers]{Ramanujan-inspired series for $1/\pi$ \\ involving harmonic numbers}



\begin{document}

\begin{abstract}
By applying the derivative operator to the known identities  from hypergeometric series or WZ pairs,  we obtain seven  series   associated with  harmonic numbers. Specifically, six of them are Ramanujan-like formulas for $1/\pi$ and the remaining one contains  harmonic numbers of order 2. As  conclusions,  Sun's five conjectural series are proved.
\end{abstract}

\maketitle

\section{Introduction}

In 1914, Ramanujan \cite{Rama} first systematically investigated some series for $1/\pi$,  including the following beautiful ones
\begin{align}
&\sum_{k=0}^{\infty}(-1)^k\frac{(\frac{1}{2})_k^3}{(1)_k^3}(4k+1)=\frac{2}{\pi},\label{ra1}\\
&\sum_{k=0}^{\infty}\frac{1}{4^k}\frac{(\frac{1}{2})_k^3}{(1)_k^3}(6k+1)=\frac{4}{\pi},\label{ra2}\\
&\sum_{k=0}^{\infty}\frac{(\frac{1}{2})_k(\frac{1}{4})_k(\frac{3}{4})_k}{(-4)^k(1)_k^3}(20k+3)=\frac{8}{\pi},\label{ra3}\\
&\sum_{k=0}^{\infty}\frac{1}{64^k}\frac{(\frac{1}{2})_k^3}{(1)_k^3}(42k+5)=\frac{16}{\pi}\label{ra4}.
\end{align}
Here and in what follows, the Pochhammer symbol is defined as
\begin{align*}
(x)_n=\Gamma(x+n)/\Gamma(x),
\end{align*}
where $\Gamma(x)$ is Euler's gamma function and $n$ is a nonnegative integer.

In 2013, Guillera \cite[(32)]{Guillera} proved an elegant Ramanujan-type formula
\begin{align*}
\sum_{k=0}^{\infty}(-1)^{k}\frac{(\frac{1}{2})_k^3}{(1)_k^3}\big\{3(4k+1)H_{k}-2\big\}=-\frac{12\log2}{\pi},
\end{align*}
where $H_n$  denotes  the ordinary harmonic numbers
$$
H_0=0\quad\text{and}\quad H_n=\sum_{k=1}^{n}\frac{1}{k}.
$$
For a complex number $x$ and a positive integer $l$, the generalized  harmonic numbers of order $l$ are defined as
$$
H_0^{(l)}(x)=0\quad\text{and}\quad H_n^{(l)}(x)=\sum_{k=1}^{n}\frac{1}{(x+k)^l}.
$$
In \cite{sun}, Sun  conjectured hundreds of series and congruences with summands containing harmonic numbers, such as
\begin{align*}
\sum_{k=0}^{\infty}\frac{(\frac{1}{2})_k^3(\frac{1}{4})_k(\frac{3}{4})_k}{16^k(1)_k^5}\big\{(120k^2+34k+3)(H_{2k}-2H_k)+68k+9\big\}=\frac{128\log2}{\pi^2},
\end{align*}
which was recently confirmed by Wei \cite{wei3} through the derivative operator method. It is an interesting topic to evaluate summations involving harmonic numbers. Apart from \cite{Guillera} and \cite{sun}, more mathematical literature concerning harmonic sums  can be found in \cite{Chen, wei, Cam1, Cam2, Guillera2, KNB, weiG}.

Motivated by the work forementioned and the formulas \eqref{ra1}--\eqref{ra4}, we establish six  Ramanujan-type representations for $1/\pi$ involving harmonic numbers as follows.  We see that Sun's four conjectural series  \cite[Equations (3.52), (3.53), (3.77) and a special case of Conjecture 3.12 (\romannumeral1)]{sun} are verified.

\begin{theorem}\label{Theorem4}
\begin{align}\label{5}
\sum_{k=0}^{\infty}(-1)^{k}\frac{(\frac{1}{2})_k^3}{(1)_k^3}(4k+1)H_{2k}=-\frac{2\log2}{\pi}.
\end{align}
\end{theorem}

\begin{theorem}\label{Theorem1}
\begin{align}\label{1}
\sum_{k=0}^{\infty}\frac{1}{4^k}\frac{(\frac{1}{2})_k^3}{(1)_k^3}\big\{(6k+1)H_k-2\big\}&=-\frac{8\log 2}{\pi}.
\end{align}
\end{theorem}

\begin{theorem}\label{Theorem5}
\begin{align}\label{2}
\sum_{k=0}^{\infty}\frac{1}{4^k}\frac{(\frac{1}{2})_k^3}{(1)_k^3}\big\{(6k+1)H_{2k}-1\big\}&=-\frac{8\log 2}{3\pi}.
\end{align}
\end{theorem}


\begin{theorem}\label{Theorem2}
\begin{align}\label{3}
\sum_{k=0}^{\infty}\frac{(\frac{1}{2})_k(\frac{1}{4})_k(\frac{3}{4})_k}{(-4)^k(1)_k^3}\big\{(20k+3)(H_{2k}-3H_k)+12\big\}=\frac{56\log 2}{\pi}.
\end{align}
\end{theorem}

\begin{theorem}\label{Theorem7}
\begin{align}\label{H1}
\sum_{k=0}^{\infty}\frac{(\frac{1}{2})_k(\frac{1}{4})_k(\frac{3}{4})_k}{(-4)^k(1)_k^3}\big\{(20k+3)(2H_{4k}-H_{2k}+H_k)-2\big\}=-\frac{16\log 2}{\pi}.
\end{align}
\end{theorem}

\begin{theorem}\label{Theorem3}
\begin{align}\label{4}
\sum_{k=0}^{\infty}\frac{1}{64^k}\frac{(\frac{1}{2})_k^3}{(1)_k^3}\big\{(42k+5)(H_{2k}-H_k)+7\big\}=\frac{32\log2}{\pi}.
\end{align}
\end{theorem}

From Theorems \ref{Theorem1} and \ref{Theorem5}, we obtain the following result
\begin{align*}
\sum_{k=0}^{\infty}\frac{1}{4^k}\frac{(\frac{1}{2})_k^3}{(1)_k^3}\big\{(6k+1)(H_{2k}-H_k)+1\big\}&=\frac{16\log 2}{3\pi},
\end{align*}
which has been given by Guillera (cf. \cite[Identity 1]{Guillera2}) and corresponds to another special case in \cite[Conjecture 3.12 (\romannumeral1)]{sun}.


Moreover, Sun \cite[Equation (4.37)]{sun} also proposed the following conjecture, which, to the best of our knowledge, had not  previously been confirmed.
\begin{theorem}\label{Theorem6}
\begin{align}\label{2H}
\sum_{k=1}^{\infty}\frac{(\frac{1}{2})_k^3(\frac{1}{4})_k(\frac{3}{4})_k}{16^k(1)_k^5}\big\{(120k^2+34k+3)\big(23H_{2k}^{(2)}-7H_k^{(2)}\big)+24\big\}=\frac{16}{3}.
\end{align}
\end{theorem}

As a prerequisite, we   review some basic concepts.
Given a differentiable function $f(x)$,  the derivative operator $\mathcal{D}_x$ is defined as
$$
\mathcal{D}_xf(x)=\frac{\text{d}}{\text{d}x}f(x).
$$
Besides,
the digamma function is given by
\begin{align*}
\psi(x)=\mathcal{D}_x\big(\log \Gamma(x)\big)=
-\gamma+\sum_{k=0}^{\infty}\bigg\{\frac{1}{k+1}-\frac{1}{k+x}\bigg\},
\end{align*}
where $\gamma$ stands for the Euler's constant. By acting the operator $\mathcal{D}_x$
 on the above equation, we obtain
 $$
\mathcal{D}_x\psi(x):= \psi'(x)=\sum_{k=0}^{\infty}\frac{1}{(k+x)^2}.
 $$
Additionally, there is a recurrence relation about the digamma function
\begin{align*}
\psi(x+1)=\psi(x)+\frac{1}{x}
\end{align*}
and some concrete values  which will be utilized in the proof
\begin{align*}
&\psi(1)=-\gamma, \qquad\qquad\qquad\qquad\quad \psi(\tfrac{1}{2})=-\gamma-2\log2, \notag\\
&\psi(\tfrac{1}{4})=-\gamma-3\log2-\frac{\pi}{2}, \qquad\quad \psi(\tfrac{3}{4})=-\gamma-3\log2+\frac{\pi}{2},\notag\\
&\psi'(1)=\frac{\pi^2}{6}, \qquad\qquad\qquad\qquad\quad \psi'(\tfrac{1}{2})=\frac{\pi^2}{2}.
\end{align*}

The rest of this paper is  organized as follows. By  means of the derivative operator method,   Theorem \ref{Theorem4},  Theorems \ref{Theorem1}--\ref{Theorem2}, Theorems \ref{Theorem7}--\ref{Theorem3}  and Theorem \ref{Theorem6} shall be proved in Sections 2, 3, 4 and 5 respectively. All above, the derivations of Theorems \ref{Theorem7}--\ref{Theorem3} will employ WZ pairs.

\section{Proof of Theorem \ref{Theorem4}}
In  the proof of  Theorem \ref{Theorem4}, Dougall's well-poised $_5F_4$-series    \cite[cf. P. 27]{Baily}  counts for much
\begin{align}\label{F54}
&_{5}F_4\left[\begin{array}{cccccccc}
a,1+\frac{a}{2},b,c,d \\
  \frac{a}{2},1+a-b,1+a-c,1+a-d
\end{array}; 1
\right]\notag\\
&=\frac{\Gamma(1+a-b)\Gamma(1+a-c)\Gamma(1+a-d)\Gamma(1+a-b-c-d)}{\Gamma(1+a)\Gamma(1+a-b-c)\Gamma(1+a-b-d)\Gamma(1+a-c-d)},
\end{align}
where $\Re(1+a-b-c-d)>0$ and the hypergeometric series \cite[cf. \S 1.2 ]{Gasper&Rahman2014} is defined as
\begin{align*}
_{r+1}F_r\left[\begin{array}{cccccccc}
a_1,a_2,\cdots,a_{r+1} \\
  b_1,b_2,\cdots,b_r
\end{array}; z
\right]=\sum_{k=0}^{\infty}\frac{(a_1)_k(a_2)_k\cdots(a_{r+1})_k}{(1)_k(b_1)_k(b_2)_k\cdots(b_{r})_k}z^k.
\end{align*}
Arrange $(a,b,d)=(\frac{1}{2},\frac{1}{2},-\infty)$ in \eqref{F54} to obtain
\begin{align}\label{th1}
\sum_{k=0}^{\infty}(-1)^k\frac{(\frac{1}{2})_k^2(c)_k}{(1)_k^2(\frac{3}{2}-c)_k}(4k+1)
=\frac{2\Gamma(\frac{3}{2}-c)}{\sqrt{\pi}\Gamma(1-c)}.
\end{align}
Notice that \eqref{th1} has been  showed by Ekhad and Zeilberger \cite{EZ}   using the WZ method.
Draw the operator $\mathcal{D}_c$ on both sides of \eqref{th1} to attain
\begin{align*}
&\sum_{k=0}^{\infty}(-1)^k\frac{(\frac{1}{2})_k^2(c)_k}{(1)_k^2(\frac{3}{2}-c)_k}(4k+1)\big\{H_k(c-1)+H_k(\tfrac{1}{2}-c)\big\}\notag\\
&=\frac{2\Gamma(\frac{3}{2}-c)}{\sqrt{\pi}\Gamma(1-c)}\big\{\psi(1-c)-\psi(\tfrac{3}{2}-c)\big\}.
\end{align*}
Letting $c=\frac{1}{2}$ in the above equation, we arrive at Theorem \ref{Theorem1}.

\section{Proofs of Theorems \ref{Theorem1}--\ref{Theorem2}}
The hypergeometric transformation formula due to Chu and Zhang \cite[Theorem 9]{ChuZhang} may contribute to the  derivations of Theorems \ref{Theorem1}--\ref{Theorem2}:
\begin{align}\label{Chu}
&\sum_{k=0}^{\infty}\frac{(c)_k(d)_k(e)_k(1+a-b-c)_k(1+a-b-d)_k(1+a-b-e)_k}{(1+a-c)_k(1+a-d)_k(1+a-e)_k(1+2a-b-c-d-e)_k}\notag\\
&\quad\times \frac{(-1)^k}{(1+a-b)_{2k}}\alpha_k(a,b,c,d,e)\notag\\
&=\sum_{k=0}^{\infty}\frac{(a+2k)(b)_k(c)_k(d)_k(e)_k}{(1+a-b)_k(1+a-c)_k(1+a-d)_k(1+a-e)_k},
\end{align}
where $\Re(1+2a-b-c-d-e)>0$  and
\begin{align*}
\alpha_k(a,b,c,d,e)=&\frac{(1+2a-b-c-d+2k)(a-e+k)}{1+2a-b-c-d-e+k}\notag\\
&+\frac{(e+k)(1+a-b-c+k)(1+a-b-d+k)}{(1+a-b+2k)(1+2a-b-c-d-e+k)}.
\end{align*}

\begin{proof}[Proof of Theorem \ref{Theorem1}]

Select $(a,b,d)=(\frac{1}{2},\frac{1}{2},\frac{1}{2})$ in \eqref{Chu} and evaluate Dougall's theorem \eqref{F54} to acquire
\begin{align}\label{1c1}
&\sum_{k=0}^{\infty}\frac{(-1)^k(\frac{1}{2})_k^2(c)_k(1-c)_k(e)_k(1-e)_k}{(1)_k(1)_{2k}(\frac{3}{2}-c)_k(\frac{3}{2}-e)_k(1-c-e)_k}\alpha_k(\tfrac{1}{2},\tfrac{1}{2},c,\tfrac{1}{2},e)\notag\\
&=\frac{\Gamma(\frac{3}{2}-c)\Gamma(\frac{3}{2}-e)\Gamma(1-c-e)}{\sqrt{\pi}\Gamma(1-c)\Gamma(1-e)\Gamma(\frac{3}{2}-c-e)}.
\end{align}
 For  a positive integer   $n$, taking $e=-n\rightarrow -\infty$  in \eqref{1c1} and observing the property (cf. \cite[Eq.(1.4.3)]{specialfunction})
\begin{align*}
\lim_{n\rightarrow \infty}\frac{\Gamma(x+n)}{\Gamma(y+n)}n^{y-x}=1,
\end{align*}
we are led to
\begin{align}\label{1c2}
\sum_{k=0}^{\infty}\frac{1}{4^k}\frac{(\frac{1}{2})_k(c)_k(1-c)_k}{(1)_k^2(\frac{3}{2}-c)_k}\frac{3k-c+1}{2}=\frac{\Gamma(\frac{3}{2}-c)}{\sqrt{\pi}\Gamma(1-c)}.
\end{align}
Apply the operator $\mathcal{D}_c$ on both sides of \eqref{1c2} to gain
\begin{align}\label{1c3}
&\sum_{k=0}^{\infty}\frac{1}{4^k}\frac{(\frac{1}{2})_k(c)_k(1-c)_k}{(1)_k^2(\frac{3}{2}-c)_k}\bigg\{\frac{3k-c+1}{2}\big(H_k(c-1)-H_k(-c)+H_k(\tfrac{1}{2}-c)\big)-\frac{1}{2}\bigg\}\notag\\
&=\frac{\Gamma(\frac{3}{2}-c)}{\sqrt{\pi}\Gamma(1-c)}\big\{\psi(1-c)-\psi(\tfrac{3}{2}-c)\big\}.
\end{align}
The $c=\frac{1}{2}$ case of \eqref{1c3} produces the desired result \eqref{1}.
\end{proof}

\begin{proof}[Proof of Theorem \ref{Theorem5}]

Choose $(a,c,d,e)=(\frac{1}{2},\frac{1}{2},\frac{1}{2},-\infty)$ in \eqref{Chu} to get
\begin{align}\label{2b2}
\sum_{k=0}^{\infty}\frac{(\frac{1}{2})_k^2(1-b)_k^2}{(1)_k^2(\frac{3}{2}-b)_{2k}}f_k(b)=\frac{\Gamma(\frac{3}{2}-b)}{\sqrt{\pi}\Gamma(1-b)},
\end{align}
where
$$
f_k(b)=\frac{6k^2-(4b-6)k-b+1}{4k-2b+3}.
$$
Adopt the operator $\mathcal{D}_b$ on both sides of \eqref{2b2} to draw
\begin{align}\label{2b3}
&\sum_{k=0}^{\infty}\frac{(\frac{1}{2})_k^2(1-b)_k^2}{(1)_k^2(\frac{3}{2}-b)_{2k}}\big\{f_k(b)\big(H_{2k}(\tfrac{1}{2}-b)-2H_k(-b)\big)-\mathcal{D}_bf_k(b)\big\}\notag\\
&=\frac{\Gamma(\frac{3}{2}-b)}{\sqrt{\pi}\Gamma(1-b)}\big\{\psi(1-b)-\psi(\tfrac{3}{2}-b)\big\}.
\end{align}
The $b=\frac{1}{2}$ case of \eqref{2b3} gives
\begin{align}\label{2b4}
\sum_{k=0}^{\infty}\frac{1}{4^k}\frac{(\frac{1}{2})_k^3}{(1)_k^3}\big\{(6k+1)(3H_{2k}-2H_k)+1\big\}=\frac{8\log 2}{\pi}.
\end{align}
The combination of \eqref{1} and \eqref{2b4} yields the wanted result \eqref{2}.
\end{proof}

\begin{proof}[Proof of Theorem \ref{Theorem2}]
Substitute $(a,c,d,e)=(\frac{1}{2},\frac{1}{2},\frac{1}{4},\frac{3}{4})$ in \eqref{Chu} and multiply both sides by $\frac{1}{2}-b$ to get
\begin{align}\label{3b1}
\sum_{k=0}^{\infty}(-1)^k\frac{(\frac{1}{2})_k(\frac{3}{4}-b)_k(1-b)_k(\frac{5}{4}-b)_k}{(1)_k(\frac{3}{2}-b)_{2k}(\frac{3}{2}-b)_{k}}g_k(b)=\frac{\sqrt{2}\Gamma(\frac{3}{2}-b)^2}{4\Gamma(\frac{3}{4}-b)\Gamma(\frac{5}{4}-b)},
\end{align}
where
\begin{align*}
g_k(b)=\frac{40k^2+(50-48b)k+16b^2-32b+15}{16(4k-2b+3)}.
\end{align*}
Utilize the operator $\mathcal{D}_b$ on both sides of \eqref{3b1} to obtain
\begin{align}\label{3b2}
&\sum_{k=0}^{\infty}(-1)^k\frac{(\frac{1}{2})_k(\frac{3}{4}-b)_k(1-b)_k(\frac{5}{4}-b)_k}{(1)_k(\frac{3}{2}-b)_{2k}(\frac{3}{2}-b)_{k}}\notag\\
&\quad\times\big\{g_k(b) \big(H_{2k}(\tfrac{1}{2}-b)+H_{k}(\tfrac{1}{2}-b)-H_k(-\tfrac{1}{4}-b)-H_k(-b)-H_k(\tfrac{1}{4}-b)\big)+\mathcal{D}_bg_k(b)\big\}\notag\\
&=\frac{\sqrt{2}\Gamma(\frac{3}{2}-b)^2}{4\Gamma(\frac{3}{4}-b)\Gamma(\frac{5}{4}-b)}\big\{\psi(\tfrac{3}{4}-b)+\psi(\tfrac{5}{4}-b)-2\psi(\tfrac{3}{2}-b)\big\}.
\end{align}
Based on the formula $H_k(-\frac{1}{4})+H_k(-\frac{3}{4})=4H_{4k}-2H_{2k}$, the $b\rightarrow \frac{1}{2}^-$ case of \eqref{3b2} turns to be
\begin{align}\label{3b3}
&\sum_{k=0}^{\infty}\frac{(\frac{1}{2})_k(\frac{1}{4})_k(\frac{3}{4}-b)_k}{(-4)^k(1)_{k}^3}\bigg\{(20k+3)(H_{2k}+2H_k-4H_{4k})-14+\frac{1}{2k+1}\bigg\}\notag\\
&=\frac{8}{\pi}\big\{\psi(\tfrac{1}{4})+\psi(\tfrac{3}{4})-2\psi(1)\big\}.
\end{align}

Opt $(a,b,d,e)=(\frac{1}{2},\frac{1}{2},\frac{1}{4},\frac{3}{4})$ in \eqref{Chu} to get
\begin{align}\label{4c1}
\sum_{k=0}^{\infty}(-1)^k\frac{(\frac{1}{4})_k(\frac{3}{4})_k(c)_k(1-c)_k}{(1)_{2k}(\frac{3}{2}-c)_k^2}h_k(c)=\frac{\sqrt{2}\Gamma(\frac{3}{2}-c)^2}{4\Gamma(\frac{3}{4}-c)\Gamma(\frac{5}{4}-c)},
\end{align}
where
$$
h_k(c)=\frac{20k^2+(19-12c)k-5c+4}{16(2k+1)}.
$$
Employ the operator $\mathcal{D}_c$ on both sides of \eqref{4c1} to gain
\begin{align}\label{4c2}
&\sum_{k=0}^{\infty}(-1)^k\frac{(\frac{1}{4})_k(\frac{3}{4})_k(c)_k(1-c)_k}{(1)_{2k}(\frac{3}{2}-c)_k^2}\notag\\
&\quad\times \big\{h_k(c) \big(H_k(c-1)-H_k(-c)+2H_{k}(\tfrac{1}{2}-c)\big)+\mathcal{D}_ch_k(c)\big\}\notag\\
&=\frac{\sqrt{2}\Gamma(\frac{3}{2}-c)^2}{4\Gamma(\frac{3}{4}-c)\Gamma(\frac{5}{4}-c)}\big\{\psi(\tfrac{3}{4}-c)+\psi(\tfrac{5}{4}-c)-2\psi(\tfrac{3}{2}-c)\big\}.
\end{align}
The $c\rightarrow \frac{1}{2}^-$ case of \eqref{4c2} becomes
\begin{align}\label{4c3}
&\sum_{k=0}^{\infty}\frac{(\frac{1}{2})_k(\frac{1}{4})_k(\frac{3}{4}-b)_k}{(-4)^k(1)_{k}^3}\bigg\{(20k+3)H_k-6+\frac{1}{2k+1}\bigg\}\notag\\
&=\frac{4}{\pi}\big\{\psi(\tfrac{1}{4})+\psi(\tfrac{3}{4})-2\psi(1)\big\}.
\end{align}

Change $(a,b,c,e)=(\frac{1}{2},\frac{1}{2},\frac{1}{2},\frac{3}{4})$ in \eqref{Chu} to attain
\begin{align}\label{5d1}
&\sum_{k=0}^{\infty}\frac{(-1)^k(\frac{1}{2})_k^2(\frac{1}{4})_k(d)_k(1-d)_k}{(1)_k(1)_{2k}(\frac{3}{2}-d)_{k}(\frac{5}{4}-d)_{k}}s_k(d)=\frac{\Gamma(\frac{3}{4})\Gamma(\frac{3}{2}-d)\Gamma(\frac{5}{4}-d)}{\sqrt{\pi}\Gamma(\frac{1}{4})\Gamma(1-d)\Gamma(\frac{3}{4}-d)},
\end{align}
where
$$
s_k(d)=\frac{1}{8}\big(20k^2+(11-12d)k-d+1\big).
$$
Apply the operator $\mathcal{D}_d$ on both sides of \eqref{5d1} to acquire
\begin{align}\label{5d2}
&\sum_{k=0}^{\infty}\frac{(-1)^k(\frac{1}{2})_k^2(\frac{1}{4})_k(d)_k(1-d)_k}{(1)_k(1)_{2k}(\frac{3}{2}-d)_{k}(\frac{5}{4}-d)_{k}}\notag\\
&\quad\times\big\{s_k(d)\big(H_k(d-1)-H_k(-d)+H_{k}(\tfrac{1}{2}-d)+H_{k}(\tfrac{1}{4}-d)\big)+\mathcal{D}_ds_k(d)\big\}\notag\\
&=\frac{\Gamma(\frac{3}{4})\Gamma(\frac{3}{2}-d)\Gamma(\frac{5}{4}-d)}{\sqrt{\pi}\Gamma(\frac{1}{4})\Gamma(1-d)\Gamma(\frac{3}{4}-d)}\big\{\psi(\tfrac{3}{4}-d)+\psi(1-d)-\psi(\tfrac{3}{2}-d)-\psi(\tfrac{5}{4}-d)\big\}.
\end{align}
The $d\rightarrow \frac{1}{4}^-$ case of \eqref{5d2} transforms into
\begin{align}\label{5d3}
&\sum_{k=0}^{\infty}\frac{(\frac{1}{2})_k(\frac{1}{4})_k(\frac{3}{4})_k}{(-4)^k(1)_k^3}\big\{(20k+3)\big(2H_{k}(-\tfrac{3}{4})-H_k(-\tfrac{1}{4})+H_k-4\big)+8\big\}\notag\\
&=\frac{8}{\pi}\big\{\psi(\tfrac{1}{2})+\psi(\tfrac{3}{4})-\psi(\tfrac{5}{4})-\psi(1)\big\}.
\end{align}

Replace $(a,b,c,d)=(\frac{1}{2},\frac{1}{2},\frac{1}{2},\frac{1}{4})$ in \eqref{Chu} to get
\begin{align}\label{6e1}
&\sum_{k=0}^{\infty}(-1)^k\frac{(\frac{1}{2})_k^2(\frac{3}{4})_k(e)_k(1-e)_k}{(1)_k(1)_{2k}(\frac{3}{2}-e)_{k}(\frac{7}{4}-e)_{k}}t_k(e)=\frac{\Gamma(\frac{1}{4})\Gamma(\frac{3}{2}-e)\Gamma(\frac{7}{4}-e)}{4\sqrt{\pi}\Gamma(\frac{3}{4})\Gamma(1-e)\Gamma(\frac{5}{4}-e)},
\end{align}
where
$$
t_k(e)=\frac{1}{8}(5k-3e+3).
$$
Conduct the operator $\mathcal{D}_e$ on both sides of \eqref{6e1} to arrive at
\begin{align}\label{6e2}
&\sum_{k=1}^{\infty}(-1)^k\frac{(\frac{1}{2})_k^2(\frac{3}{4})_k(e)_k(1-e)_k}{(1)_k(1)_{2k}(\frac{3}{2}-e)_{k}(\frac{7}{4}-e)_{k}}\notag\\
&\quad\times\big\{t_k(e)\big(H_k(e-1)-H_k(-e)+H_{k}(\tfrac{1}{2}-e)+H_{k}(\tfrac{3}{4}-e)\big)+\mathcal{D}_et_k(e)\big\}\notag\\
&=\frac{\Gamma(\frac{1}{4})\Gamma(\frac{3}{2}-e)\Gamma(\frac{7}{4}-e)}{4\sqrt{\pi}\Gamma(\frac{3}{4})\Gamma(1-e)\Gamma(\frac{5}{4}-e)}\big\{\psi(1-e)+\psi(\tfrac{5}{4}-e)-\psi(\tfrac{3}{2}-e)-\psi(\tfrac{7}{4}-e)\big\}.
\end{align}
The $e\rightarrow \frac{3}{4}^-$ case of \eqref{6e2} concludes
\begin{align}\label{6e3}
&\sum_{k=0}^{\infty}\frac{(\frac{1}{2})_k(\frac{1}{4})_k(\frac{3}{4})_k}{(-4)^k(1)_k^3}\big\{(20k+3)\big(2H_{k}(-\tfrac{1}{4})-H_k(-\tfrac{3}{4})+H_k\big)-12\big\}\notag\\
&=\frac{8}{\pi}\big\{\psi(\tfrac{1}{4})+\psi(\tfrac{1}{2})-\psi(\tfrac{3}{4})-\psi(1)\big\}.
\end{align}
The result of the linear combination $\eqref{4c3}-\eqref{3b3}-\eqref{5d3}-\eqref{6e3}$ coincides with \eqref{3} completely. Hence, we finish the proof of this theorem.
\end{proof}

\section{Proofs of Theorems \ref{Theorem7} and \ref{Theorem3}}
Different from the precious proofs, in this section, we will carry out the derivative on  the identities that are formulated by WZ pairs.
\begin{proof}[Proof of Theorem \ref{Theorem7}]
By setting up suitable WZ pairs $(F(n,k),G(n,k))$,
 Guillera \cite{Guillera2} showed that several identities involving $\pi$ can be extended to the ones with an extra parameter. For example,
for such WZ pair
\begin{align}\label{WZ}
\begin{aligned}
	& F(n, k)=\frac{64}{\pi^3} \frac{n^2}{4 n-2 k-1} \frac{\cos (\pi k) \Gamma(2 n-k+1 / 2) \Gamma(n+1 / 2)^3 \Gamma(k+1 / 2)^2}{(-1)^n\Gamma(n+k+1) \Gamma(2 n+1)^2}, \\
	& G(n, k)=\frac{1}{\pi^3}(20 n+2 k+3) \frac{ \cos (\pi k) \Gamma(2 n-k+1 / 2) \Gamma(n+1 / 2)^3 \Gamma(k+1 / 2)^2}{(-1)^n\Gamma(n+k+1) \Gamma(2 n+1)^2} ,
\end{aligned}
\end{align}
Guillera \cite{Guillera2} obtained, for any $k$,
\begin{align}\label{G}
\sum_{n=0}^\infty G(n,k) = \frac{8}{\pi}.
\end{align}
 In \eqref{G}, taking the derivative with respect to $k$ and then setting $k=0$, we derive that
\begin{align*}
\sum_{n=0}^\infty (-1)^n \frac{{4n \choose 2n} {2n \choose n}^2}{2^{10n}}
\big\{ (20n+3)  (2 H_{4n} - H_{2n} +H_n + 2 \log 2) - 2\big\} = 0.
\end{align*}
Thus, by \eqref{ra3} and the above equation, we prove the truth of Theorem \ref{Theorem7}.
\end{proof}
\begin{proof}[Proof of Theorem \ref{Theorem3}]
Let
\[
H(n,k) = F(n+1,n+k+1) + G(n,n+k),
\]
where $F(n,k)$ and  $G(n,k)$ are defined in \eqref{WZ}.
Guillera \cite{Guillera2} also presented the following result
\[
\sum_{n=0}^\infty H(n,k) = \frac{8}{\pi}.
\]
Now taking the derivative with respect to $k$ and then setting $k=0$, we derive that
\begin{align*}
\sum_{n=0}^\infty \frac{{2n \choose n}^3}{2^{12n}} \big\{(42n+5) ( H_{2n} - H_n - 2 \log 2) + 7\big\}= 0.
\end{align*}
Hence, by \eqref{ra4} and the above identity, we complete the proof  of  Theorem  \ref{Theorem3}.
\end{proof}

\section{Proof of Theorem \ref{Theorem6}}
Recall that another hypergeometric transformation  \cite[Theorem 14]{ChuZhang} is stated as
\begin{align}\label{ChuTh14}
&\sum_{k=0}^{\infty}\frac{(c)_k(e)_k(1+a-b-c)_k(1+a-b-e)_k(1+a-c-d)_k(1+a-d-e)_k}{(1+a-b)_{2k}(1+a-d)_{2k}(1+2a-b-c-d-e)_{2k}}\notag\\
&\quad\times \frac{(1+a-b-d)_{2k}}{(1+a-c)_{k}(1+a-e)_{k}}\delta_k(a,b,c,d,e)\notag\\
&=\sum_{k=0}^{\infty}\frac{(a+2k)(b)_k(c)_k(d)_k(e)_k}{(1+a-b)_k(1+a-c)_k(1+a-d)_k(1+a-e)_k},
\end{align}
where $\Re(1+2a-b-c-d-e)>0$  and
\begin{align*}
\delta_k(a,b,c,d,e)=&\frac{(1+2a-b-c-d+3k)(a-e+k)}{1+2a-b-c-d-e+2k}+\frac{(e+k)(2+2a-b-d-e+3k)}{(1+a-b+2k)(1+a-d+2k)}\notag\\
&\times \frac{(1+a-b-c+k)(1+a-b-d+2k)(1+a-c-d+k)}{(1+2a-b-c-d-e+2k)(2+2a-b-c-d-e+2k)}.
\end{align*}
Select $(a,c,d,e)=(\frac{1}{2},\frac{1}{2},1-b,\frac{1}{2})$ in \eqref{ChuTh14} to obtain
\begin{align}\label{th6.1}
\sum_{k=0}^{\infty}\frac{(\frac{1}{2})_{2k}(\frac{1}{2})_k^2(b)_k^2(1-b)_k^2}{(1)_{k}^2(1)_{2k}(\frac{3}{2}-b)_{2k}(\frac{1}{2}+b)_{2k}}w_k(b)=\frac{\Gamma(\frac{3}{2}-b)\Gamma(\frac{1}{2}+b)}{\pi\Gamma(b)\Gamma(1-b)},
\end{align}
where
$$
w_k(b)=\frac{120k^4+154k^3-(48b^2-48b-55)k^2-(22b^2-22b-6)k-3b^2+3b}{2(4k-2b+3)(4k+2b+1)}.
$$
Carry out the  operator $\mathcal{D}_b$ on both sides of \eqref{th6.1} to get
\begin{align}\label{th6.2}
&\sum_{k=0}^{\infty}\frac{(\frac{1}{2})_{2k}(\frac{1}{2})_k^2(b)_k^2(1-b)_k^2}{(1)_{k}^2(1)_{2k}(\frac{3}{2}-b)_{2k}(\frac{1}{2}+b)_{2k}}\notag\\
&\quad\times \big\{w_k(b)\big(2H_k(b-1)-2H_k(-b)+H_{2k}(\tfrac{1}{2}-b)-H_{2k}(b-\tfrac{1}{2})\big)+\mathcal{D}_bw_k(b)\big\}\notag\\
&=\frac{\Gamma(\frac{3}{2}-b)\Gamma(\frac{1}{2}+b)}{\pi\Gamma(b)\Gamma(1-b)}\big\{\psi(\tfrac{1}{2}+b)-\psi(\tfrac{3}{2}-b)+\psi(1-b)-\psi(b)\big\}.
\end{align}
Divide both sides of \eqref{th6.2} by $1-2b$ to get
\begin{align}\label{th6.3}
&\sum_{k=0}^{\infty}\frac{(\frac{1}{2})_{2k}(\frac{1}{2})_k^2(b)_k^2(1-b)_k^2}{(1)_{k}^2(1)_{2k}(\frac{3}{2}-b)_{2k}(\frac{1}{2}+b)_{2k}}\notag\\
&\quad\times \bigg\{w_k(b)\bigg(\sum_{j=1}^{k}\frac{2}{(b-1+j)(-b+j)}+\sum_{j=1}^{k}\frac{1}{(\frac{1}{2}-b+j)(b-\frac{1}{2}+j)}\bigg)+\frac{\mathcal{D}_bw_k(b)}{1-2b}\bigg\}\notag\\
&=\frac{\Gamma(\frac{3}{2}-b)\Gamma(\frac{1}{2}+b)}{\pi\Gamma(b)\Gamma(1-b)}\cdot\frac{\psi(\tfrac{1}{2}+b)-\psi(\tfrac{3}{2}-b)+\psi(1-b)-\psi(b)}{1-2b}.
\end{align}
The $b\rightarrow \frac{1}{2}^-$ case of \eqref{th6.3} turns into
\begin{align}\label{th6.4}
\sum_{k=0}^{\infty}\frac{1}{16^k}\frac{(\frac{1}{2})_k^3(\frac{1}{4})_k(\frac{3}{4})_k}{(1)_k^5}\bigg\{(120k^2+34k+3)\big(7H_{2k}^{(2)}-2H_k^{(2)}\big)+18-\frac{9}{2k+1}\bigg\}=\frac{32}{3}.
\end{align}

Organise $(a,b,d,e)=(\frac{1}{2},\frac{1}{2},\frac{1}{2},1-c)$ in \eqref{ChuTh14} to obtain
\begin{align}\label{th6.11}
\sum_{k=0}^{\infty}\frac{(\frac{1}{2})_{2k}(c)_k^3(1-c)_k^3}{(1)_{2k}^3(\frac{3}{2}-c)_{k}(\frac{1}{2}+c)_{k}}z_k(c)=\frac{\Gamma(\frac{3}{2}-c)\Gamma(\frac{1}{2}+c)}{\pi\Gamma(c)\Gamma(1-c)},
\end{align}
where
\begin{align*}
z_k(c)=&\frac{1}{2(1+2k)^3}\big(60k^5+107k^4+(8c^2-8c+74)k^3\notag\\
&+(6c^2-6c+24)k^2-(4c^4-8c^3+6c^2-2c-3)k-c^4+2c^3-2c^2+c\big).
\end{align*}
Carry out the  operator $\mathcal{D}_c$ on both sides of \eqref{th6.11} to get
\begin{align}\label{th6.22}
&\sum_{k=0}^{\infty}\frac{(\frac{1}{2})_{2k}(c)_k^3(1-c)_k^3}{(1)_{2k}^3(\frac{3}{2}-c)_{k}(\frac{1}{2}+c)_{k}}\notag\\
&\quad\times \big\{z_k(c)\big(3H_k(c-1)-3H_k(-c)+H_{k}(\tfrac{1}{2}-c)-H_{k}(c-\tfrac{1}{2})\big)+\mathcal{D}_cz_k(c)\big\}\notag\\
&=\frac{\Gamma(\frac{3}{2}-c)\Gamma(\frac{1}{2}+c)}{\pi\Gamma(c)\Gamma(1-c)}\big\{\psi(\tfrac{1}{2}+c)-\psi(\tfrac{3}{2}-c)+\psi(1-c)-\psi(c)\big\}.
\end{align}
Dividing  both sides of \eqref{th6.22} by $1-2c$ and then
letting  $c\rightarrow \frac{1}{2}^-$, we have
\begin{align}\label{th6.33}
\sum_{k=0}^{\infty}\frac{1}{16^k}\frac{(\frac{1}{2})_k^3(\frac{1}{4})_k(\frac{3}{4})_k}{(1)_k^5}\bigg\{(120k^2+34k+3)\big(12H_{2k}^{(2)}-4H_k^{(2)}\big)-16+\frac{24}{2k+1}\bigg\}=\frac{32}{3}.
\end{align}
Hence, the linear combination  $8\times\eqref{th6.4}+3\times\eqref{th6.33}$ is just the result \eqref{2H}.

\end{document}